\documentclass[preprint,authoryear,a4paper]{amsart}
\usepackage{mathrsfs}
\usepackage{amssymb,url,color}
\usepackage{amsfonts}
\usepackage{amsmath,amsthm,xspace,enumerate,amscd,graphicx}
\usepackage[pdftex]{hyperref}
\usepackage[active]{srcltx}
\newtheorem{theorem}{Theorem}
\newtheorem{remark}[theorem]{Remark}
\newtheorem{lemma}[theorem]{Lemma}
\newtheorem{proposition}[theorem]{Proposition}

\renewcommand{\a}{\ensuremath{\alpha}}

\newcommand{\N}{\mathbb{N}}
\newcommand{\Z}{\ensuremath{\mathbb{Z}}\xspace}

\renewcommand{\phi}{\varphi}

\renewcommand{\leq}{\leqslant}
\renewcommand{\geq}{\geqslant}

\def\M{\mathcal{M}}
\def\P{\mathcal{P}}

\newcommand{\F}{\mathbb{F}}
\newcommand{\ts}{\times}

\begin{document}

\title
{The probability of rectangular unimodular matrices over
$\F_q[x]$ }

\author[X. Guo]{Xiangqian Guo}

\address[X. Guo]{Department of Mathematics,
Zhengzhou University,  Henan Province, 450001,
China.}
\email{\href{mailto: X. Guo
<xqguo@zzu.edu.cn>}{guoxq@zzu.edu.cn}}

\author[G. Yang]{Guangyu Yang}

\address[G. Yang]{Department of Mathematics, Zhengzhou University,  Henan Province, 450001,
China.}\email{\href{mailto: G. Yang
<study\_yang@yahoo.com.cn>}{study\_yang@yahoo.com.cn}}

\begin{abstract}
In this note, we show that the probability that
a uniformly random $k\times n$ matrix over $\F_q[x]$
can be extended to an $n\times n$ invertible matrix
is
$(1-q^{k-n})(1-q^{k-1-n})\cdots(1-q^{1-n})$.
Connections with Dirichlet's density theorem on co-prime
integers and its various generalizations are also presented.



%
\end{abstract}



\maketitle

\section{introduction and main results}

For any set $A$ of positive integers, its {\bf natural density} is
defined as
\begin{align}
D(A):=\lim_{N\rightarrow\infty}\frac{|A\cap\{1,2,\ldots,N\}|}{N},
\end{align}
provided the limit exists, where $|\cdot|$ denotes the cardinality
of the corresponding set. Dirichlet \cite{Di} discovered an
interesting density theorem that asserts the probability that two
integers are co-prime is $6/\pi^2$, that is,
\begin{align}\label{dirichlet}
\lim_{N\rightarrow\infty}\frac{|\{(m,n)\in\mathbb{N}^2:\, 1\leq
m,n\leq N, \gcd(m,n)=1\}|}{N^2}=\zeta(2)^{-1}=\frac{6}{\pi^2},
\end{align}
where $\gcd(m,n)$ denotes the greatest common divisor of $m$ and
$n$, and $\zeta(s)$ is the Riemann's zeta function. Moreover, the
probability that $n$ integers are co-prime is
\begin{align}
 \lim_{N\rightarrow\infty}\frac{|\{(m_1,\ldots,m_n)\in\mathbb{N}^n:\,
1\leq m_1,\ldots,m_n\leq
N,\gcd(m_1,\ldots,m_n)=1\}|}{N^n}=\zeta(n)^{-1}.
\end{align}
Kubota and Sugita \cite{KuSu} gave a rigorous probabilistic interpretation to
Dirichlet's density theorem. For the deep links between probability theory and number theory, please refer to Tenenbaum \cite{Te}, Kubilius \cite{Ku} and Kac \cite{Kac}.

Denote by $M_{k\times n}(R)$ the set of all $k\times n$ matrices
over a ring $R$. A matrix $A\in M_{k\times n}(R)$ is called {\bf
unimodular} if it can be extended to an $n\times n$ invertible
matrix. Note that, for the commutative ring $R$ with $1$, this is equivalent to saying that $A$ can be extended to an $n\times n$ matrix with determinant $1$ in case $k<n$.
Using the concept of unimodular, we can give a re-statement
of Dirichlet's density theorem and its generalization: the
probability that a random $1\times n$ integer matrix is unimodular
is $\zeta(n)^{-1}$. Naturally, we can consider the matrix form of
Dirichlet's density theorem: the probability that a random
$k\times n\,(1\leq k\leq n)$ integer matrix is unimodular? The key
is to define the analogous {\it natural density} for the sets of the
integer matrices. Maze, Rosenthal and Wagner \cite{MaRoWa} studied
this problem and obtained that the probability that a random
$k\times n$ matrix is unimodular is
$\prod_{j=n-k+1}^{n}\zeta(j)^{-1}$. Furthermore, \cite{MaRoWa} also
gave some interesting historical remarks for Dirichlet's density
theorem.


It is natural, of course, to ask
the similar question for matrices over
polynomial rings of fields. As remarked in \cite{MaRoWa}, the
concept of natural density does not extend naturally to the ring
$\F[x]$ for a general field $\F$. Fortunately, this can be done for
the finite fields.

Let $\F_q$ be a finite field consisting of $q$ elements, and
$\F_q[x]$ be the polynomial ring over $\F_q$. To enumerate
$\F_q[x]$, let $\Sigma$ be the set of all vectors
$\a=(a_0,a_1,\cdots)$ with $a_i\in\{0,1,\cdots,q-1\}$ and $a_i=0$
for sufficiently large $i$. Then there is a one-to-one map $\chi:
\Sigma \rightarrow \Z_+=\N\cup\{0\}$ defined by
$\chi(a_0,a_1,\cdots)=\sum_{i=0}^{\infty}a_iq^i$. For convenience,
we denote $\a_m=\chi^{-1}(m)$ and $m_\a=\chi(\a)$. Then for all
$m\in\Z_+$, we set
\[
f_m(x):=\sum_{i=0}^{\infty}a_ix^{i},\;\;{\rm
with}\;\;\a_m=(a_0,a_1,\cdots).
\]

By a probabilistic method, Sugita and Takanobu \cite{SuTa}
determined the probability that two polynomials over $\F_q$ are
co-prime {\it when $q$ is a prime}, that is,
\begin{align}\label{SuTa}
\lim_{N\rightarrow\infty}\frac{|\{(m,n)\in\{0,1,\ldots,N-1\}^2:\,
\gcd(f_m,f_n)=1\}|}{N^2}=1-\frac{1}{q}.
\end{align}
Recently, several authors considered the question of the probability
of $n$ polynomials to be co-prime in $\F_q[x]$. Morrison \cite{Mo}
introduced a concept of {\it natural density}
\begin{align}\label{Mo density}
\widetilde{D}(S):=\lim_{d\rightarrow \infty}\frac{|S\cap \mathcal{F}_d|}{|\mathcal{F}_d|},\ \ \ \ \ \ \ \ S\subseteq\F_q[x],
\end{align}
where $\mathcal{F}_d$ is the set of all the polynomials in $\F[x]$
with degree at most $d$. He first treated this problem in a
heuristic way. Then through an elegant as well as rigorous argument,
he deduced the following formula,
\begin{align}\label{Mo}
\lim_{d\rightarrow\infty}\frac{|\{(f_1,\ldots,f_n)\in\F_q^n[x]:
\gcd(f_1,\ldots,f_n)=1,\deg(f_i)\leq d,1\leq i\leq
n\}|}{|\{f\in\F_q[x]:\,\deg(f)\leq d\}|^n}=1-\frac{1}{q^{n-1}}.
\end{align}
Benjamin and Bennett \cite{BeBe} also studied this problem; their
methods are very interesting, using the
Euclidean algorithm which is one of the oldest algorithms. 
Essentially, they also got the same conclusion.




Our object in this paper is to extend the main results in
\cite{MaRoWa} to the polynomial ring $\F_q[x]$ for any prime power
$q$, that is, the matrix form of Morrison's theorem (\ref{Mo}).
Denote $\F=\F_q$ and $\F[x]=\F_q[x]$ for short.
Let $\mathcal{M}:=M_{k\times
n}(\F[x])$ be the set of all $k\times n$ matrices over $\F[x]$ and
$\mathcal{M}_N$ be the subset of $\mathcal{M}$ consisting of all
matrices with entries in $\{f_0(x), f_1(x),\cdots,f_N(x)\}$. For any
subset $S\subseteq {\mathcal M}$, we define the {\bf natural
density} of $S$ in $\M$ to be
\begin{align}\label{density}
D(S):=\lim_{N\rightarrow\infty}\frac{|S\cap
\mathcal{M}_N|}{|\mathcal{M}_N|},
\end{align}
if the above limit (\ref{density}) exists. Note that, in the special
case $k=1$, the limit in (\ref{Mo density}) is just the limit of a
subsequence in (\ref{density}) with $N=q^d$.

To describe our result, i.e., to determine the natural density of
$k\times n$ unimodular matrices over $\F[x]$, we introduce the
following \textbf{$q$-zeta function}
\begin{align}
\zeta_q(j):=\prod_{f}(1-\frac{1}{q^{j\deg(f)}})^{-1}
=\prod_{m=1}^{\infty} (1-\frac{1}{q^{jm}})^{-\varphi_m},
\end{align}
where $f$ goes through all irreducible polynomials in $\F[x]$ and
$\varphi_m$ is the number of irreducible polynomials in $\F[x]$ with
degree $m$. In the above statements and what follows, {\bf by
irreducible polynomials we always mean monic irreducible polynomials
(the constant polynomial $1$ is excluded as usual).} Note that the $q$-zeta function is analogous with the Riemann zeta function. For more information, one can see Morrison \cite{Mo}.
Now we are in
the position to give our main result.
\begin{theorem} \label{result}
Let $E$ be the set of all $k\times n$ unimodular matrices over
$\F[x]$, then the natural density of $E$ is
\begin{align}
D(E)=\prod_{j=n-k+1}^{n}\prod_{f}(1-\frac{1}{q^{j\deg(f)}})
=\prod_{j=n-k+1}^{n}\prod_{m=1}^{\infty}
(1-\frac{1}{q^{jm}})^{\varphi_m}
=\prod_{j=n-k+1}^{n}\zeta_q(j)^{-1}.
\end{align}
\end{theorem}

\begin{remark}
In other words, the probability that a $k\times n$ polynomial matrix
can be completed by an $(n-k)\times n$ matrix into an $n\times n$
invertible matrix over $\F[x]$ is
$\prod_{j=n-k+1}^{n}\zeta_q(j)^{-1}$. And via introducing the q-zeta
function $\zeta_q(s)$, we obtain a similar formula to the one in
\cite{MaRoWa} for integer matrices
\[
D(\widetilde{E})=\prod_{j=n-k+1}^{n}\zeta(j)^{-1},
\]
where $\widetilde{E}$ denotes the set of all $k\times n$ unimodular matrices
over $\Z$.
\end{remark}

\begin{remark}
There is an interesting equation
\begin{align}\label{Mo equation}
\prod_{f}(1-t^{\deg(f)})^{-1}=\sum_{l=0}^{\infty}q^lt^l
=\frac{1}{1-qt},
\end{align}
where $f$ is over all irreducible polynomials.
For an interpretation of this
equation, one can see \cite{Mo}. Putting $t=q^{-j}, j\in\N$ in (\ref{Mo
equation}), we get
\begin{align}
\zeta_q(j)^{-1}=1-\frac{1}{q^{j-1}}.
\end{align}
Hence, we obtain an explicit formula for the natural density of $E$,
\begin{align}\label{explicit}
D(E)=(1-\frac{1}{q^{n-k}})(1-\frac{1}{q^{n-k+1}})\cdots
(1-\frac{1}{q^{n-1}}),
\end{align}
and an interesting identity concerning $\{\varphi_m\}$,
\begin{align}\label{interesting}
\prod_{m=1}^{\infty}
(1-\frac{1}{q^{jm}})^{\varphi_m}=1-\frac{1}{q^{j-1}},\quad\forall
j\in\N.
\end{align}
\end{remark}

\begin{remark} In particular, if we put $k=1$, Theorem \ref{result} shows that
the probability that $n$ polynomials over the finite field $\F$ to
be co-prime is
\[
\zeta_q(n)^{-1}=\prod_{m=1}^{\infty}
(1-\frac{1}{q^{nm}})^{\varphi_m}=1-\frac{1}{q^{n-1}}.
\]
This extends Dirichlet's density theorem (\ref{dirichlet}) and
yields the same conclusion as Morrison \cite{Mo} and Benjamin and
Bennett \cite{BeBe}.
\end{remark}

\begin{remark}
As an application, our work suggests that there may be a simple probabilistic algorithm for writing a projective $\mathbb{F}_q[x]$-module as a free module, please refer to \cite{Web} for more details.
\end{remark}

We end this section with some remarks on the square case $k=n$. One
may notice that our proof of Theorem \ref{result} in \S \ref{proof}
does not apply to this case (also, $\zeta_q(1)$ is not well
defined in this case). However, we have the following result,
\begin{proposition}\label{square case}
The natural density of $n\times n$ unimodular matrices over $\F_q[x]$
is $0$.
\end{proposition}

\begin{remark}
The proof of Proposition \ref{square case} is quite simple and
similar to the proof of Lemma 5 in \cite{MaRoWa}, so we omit the
details. If we make the convention that $\zeta_q(1)^{-1}=0$, Theorem
\ref{result} also holds for $k=n$.
\end{remark}



\section{Proof of Main result}\label{proof}

In this section, we shall give the proof of Theorem \ref{result}.
Before this, we need some preparations.

Let $\mathcal{P}$ be a finite set of irreducible polynomials in
$\F[x]$ and $\widehat{\mathcal{P}}$ be the set of all irreducible
polynomials in $\F[x]$. Denote by $E_\mathcal{P}$ the set of all
matrices $A\in\mathcal{M}=M_{k\times n}(\F[x])$
such that the $\gcd$ of all full rank
minors (i.e., minors of rank $k$) is co-prime to all elements in
$\mathcal{P}$. Recall that $E$ is the set of all unimodular matrices
in $\mathcal{M}$, that is, each matrix in $E$ can be extended to
an $n\ts n$ invertible matrix. It is well known that $A\in \M$ is
unimodular if and only if the $\gcd$ of all full rank minors is $1$
(See \cite{Qu}, \cite{Su} and \cite{YoPi} for more details). Thus we have $E=\bigcap_{\mathcal{P}} E_{\mathcal{P}}$.

\begin{lemma}\label{finite}
Let $E_{\mathcal{P}}$ be defined as above, then we have
\begin{align}
D(E_{\mathcal{P}})=\prod_{j=0}^{k-1}\prod_{f\in \mathcal{P}}(1-q^{(j-n)\deg(f)})=\prod_{j=n-k+1}^{n}\prod_{f\in \mathcal{P}}(1-q^{-j\deg(f)}).
\end{align}
\end{lemma}

\begin{proof}
Denote $f_{\mathcal{P}}:=\prod_{f\in\mathcal{P}}f$ and $d_{\mathcal
P}=\deg(f_{\mathcal P})$. For any positive integer $N$, consider the
maps,
\begin{align}\label{pi}
\pi:\quad\mathcal{M}_N\longrightarrow M_{k\times
n}(\F[x]/\left(f_{\mathcal{P}})\right)
\end{align}
and
\begin{align}
\psi:\quad M_{k\times
n}(\F[x]/\left(f_{\mathcal{P}})\right)\longrightarrow M_{k\times
n}\left(\prod_{f\in \mathcal{P}}\F[x]/(f)\right)\longrightarrow
\prod_{f\in \mathcal{P}}M_{k\times n}\left(\F[x]/(f)\right),
\end{align}
where $(f_{\mathcal{P}})$ and $(f)$ denote the ideals generated by $f_\P$ and $f$ respectively. Note that $\pi$ is the canonical projection
via modulo $f_\mathcal{P}$ and $\psi$ is a composition of two
canonical isomorphisms as $\F$ vector spaces.

Taken any $A\in \M_N$. It is easy to see that $A\in E_\P$ if and
only if the component of $\psi\circ\pi(A)$ in $M_{k\times n}(\F[x]/(f))$
is full rank for each $f\in\P$. Since $f\in\P$ is irreducible, we
know that $\F[x]/(f)\simeq\F_{q^{\deg(f)}}$. Let $F_{q^{\deg(f)}}$
denote the set of all full rank $k\times n$ matrices over the finite
field $\F_{q^{\deg(f)}}$. It is well known that
$|F_{q^{\deg(f)}}|=\prod_{j=0}^{k-1}(q^{n\deg(f)}-q^{j\deg(f)})$
(See \cite{LiNi} page 455, for example).

First suppose that $N=mq^{d_\P}-1$ for some $m\in\N$. Then it is
easy to see
\[
\{f_l(x):\,0\leq l\leq N\}=\{f_s(x)x^{d_\P}+f_t(x):\,0\leq s\leq
m-1, 0\leq t\leq q^{d_\P}-1\}.
\]
For any fixed $0\leq s\leq m-1$, the following
projection is one-to-one:
\[
\{f_s(x)x^{d_\P}+f_t(x):\,0\leq t\leq q^{d_\P}-1\}\longrightarrow
\F[x]/(f_\P)
\]
and the canonical
projection
\[
\{f_l(x):\,0\leq l\leq N\}\longrightarrow \F[x]/(f_\P)
\]
is $m$-to-one. As a result, the projection map $\pi$ in (\ref{pi})
is $m^{kn}$-to-one. Thus we obtain that
\begin{align}\label{lemma1}
|E_{\mathcal{P}}\cap\mathcal{M}_N|&=m^{kn}\cdot
\prod_{f\in\mathcal{P}}|F_{q^{\deg(f)}}|\nonumber\\
&=m^{kn}\cdot\prod_{f\in\mathcal{P}} \prod_{j=0}^{k-1}q^{n\deg(f)}(1-q^{(j-n)\deg(f)})\nonumber\\
&=(mq^{d_{\mathcal{P}}})^{kn}\cdot\prod_{f\in\mathcal{P}}
\prod_{j=0}^{k-1}(1-q^{(j-n)\deg(f)})
\end{align}

Now suppose that $N$ is any positive integer. There exists some
$m,r\in\Z_+$ such that $N+1=mq^{d_\P}+r$ with $0\leq r< q^{d_\P}$.
For convenience, set $\widetilde{N}:=mq^{d_\P}-1$. Then by the definition of the natural density, we have
\begin{align}\label{lemma2}
D(E_\mathcal{P})&=\lim_{N\rightarrow\infty}
\frac{|E_{\mathcal{P}}\cap\mathcal{M}_N|}{|\mathcal{M}_N|}\nonumber\\
&=\lim_{N\rightarrow\infty}
\frac{|E_{\mathcal{P}}\cap\mathcal{M}_{\widetilde{N}}|
+|E_{\mathcal{P}}\cap(\M_N-\mathcal{M}_{\widetilde{N}})|}{(N+1)^{kn}}.
\end{align}
Note that $|\M_N-\M_{\widetilde{N}}|\leq
rkn(N+1)^{kn-1},$ which gives that
\begin{align}\label{lemma3}
\lim_{N\rightarrow\infty}\frac{|E_{\mathcal{P}}
\cap(\M_N-\mathcal{M}_{\widetilde{N}})|}{(N+1)^{kn}}=0.
\end{align}
Thus, from (\ref{lemma2}), (\ref{lemma3}) and (\ref{lemma1}), we obtain
\begin{align*}
D(E_\mathcal{P})&=\lim_{N\rightarrow\infty}
\frac{|E_{\mathcal{P}}\cap\mathcal{M}_{\widetilde{N}}|}{(N+1)^{kn}}\\
&=\lim_{N\rightarrow\infty}\frac{(N+1-r)^{kn}
\prod_{f\in\mathcal{P}}
\prod_{j=0}^{k-1}(1-q^{(j-n)\deg(f)})}{(N+1)^{kn}}\\
&=\prod_{f\in\mathcal{P}}\prod_{j=0}^{k-1}(1-q^{-(n-j)\deg(f)})\\
&=\prod_{j=n-k+1}^{n}\prod_{f\in\mathcal{P}}(1-q^{-j\deg(f)}).
\end{align*}
This completes the proof of Lemma \ref{finite}.
\end{proof}

\vskip5pt \noindent\textbf{Proof of Theorem \ref{result}.~~} Suppose that
$k<n$. For any
irreducible polynomial $f\in\F_q[x]$, let $H_f\subseteq\mathcal{M}$
be the set of all matrices whose $\gcd$ of all full rank minors is
divisible by $f$. Denote $q_f:=q^{\deg(f)}$, then by Lemma
\ref{finite},
\begin{align*}
D(H_f)&=1-D(E_{\{f\}})\\
&=1-\prod_{j=n-k+1}^{n}(1-q^{-j\deg(f)})\\
&< 1-\left(1-\sum_{j=n-k+1}^nq_f^{-j}\right)\\
&<\sum_{j=n-k+1}^{\infty}q_f^{-j}=\frac{1}{q_f^{n-k}(q_f-1)}
\leq \frac{2}{q_f^2}.
\end{align*}

Let $\mathcal{P}_t$ be the set of all irreducible polynomials with degree no
more than $t$ and denote $E_t=E_{\mathcal{P}_t}$. Since
\[
(E_{t}\setminus E)\subseteq
\bigcup_{f\in\widehat{\mathcal{P}}\setminus\mathcal{P}_t} H_f,
\]
we have
\begin{align*}
\limsup_{N\rightarrow\infty}\frac{|(E_{t}\setminus E)\cap \mathcal{M}_N|}{|\mathcal{M}_N|}&\leq\limsup_{N\rightarrow\infty}
{\frac{|(\bigcup_{f\in\widehat{\mathcal{P}}\setminus\mathcal{P}_t}
H_f)\cap\mathcal{M}_N|}{|\mathcal{M}_N|}}\\
&\leq \limsup_{N\rightarrow\infty}{\frac{
\sum_{f\in\widehat{\mathcal{P}}\setminus\mathcal{P}_t}
|H_f\cap\mathcal{M}_N|}{|\mathcal{M}_N|}}\\
&\leq \sum_{f\in\widehat{\mathcal{P}}\setminus \mathcal{P}_t}\limsup_{N\rightarrow\infty}
{\frac{|H_f\cap\mathcal{M}_N|}
{|\mathcal{M}_N|}}\\
&=\sum_{f\in\widehat{\mathcal{P}}\setminus\mathcal{P}_t}D(H_f)
\leq \sum_{f\in\widehat{\mathcal{P}}\setminus\mathcal{P}_t}
\frac{2}{q_f^2}\\
&=\sum_{f\in\widehat{\mathcal{P}}\setminus\mathcal{P}_t}
\frac{2}{q^{2\deg(f)}}=\sum_{m=t+1}^{\infty}
\frac{2\varphi_m}{q^{2m}}.
\end{align*}
It is well known that all irreducible polynomials with degree $m$ can
divide $x^{q^{m}}-x$, which has no multiple roots. Thus
$m\varphi_m\leq q^m$ and
\begin{align*}
\limsup_{N\rightarrow\infty}{\frac{|(E_{t}\setminus E)\cap
\mathcal{M}_N|}{|\mathcal{M}_N|}}
&\leq\sum_{m=t+1}^{\infty}\frac{2}{mq^{m}}\\
&\leq\sum_{m=t+1}^{\infty}\frac{2}{q^{m}}=\frac{2}{q^t(q-1)}.
\end{align*}

Now note that $E\cap\mathcal{M}_N\subseteq E_t\cap\mathcal{M}_N$ and $E\cap\mathcal{M}_N
=E_t\cap\mathcal{M}_N-(E_t\setminus E)\cap\mathcal{M}_N$, which
imply that
\begin{align*}
\liminf_{N\rightarrow\infty}\frac{|E\cap\mathcal{M}_N|}
{|\mathcal{M}_N|}
&\geq \liminf_{N\rightarrow\infty}\frac{|E_t\cap\mathcal{ M}_N|}{|\mathcal{M}_N|}-\limsup _{N\rightarrow\infty} \frac{|(E_t\setminus E)\cap\mathcal{M}_N|}{|\mathcal{M}_N|}\\
&\geq D(E_t)-\frac{2}{q^t(q-1)},
\end{align*}
and
\begin{align*}
\limsup_{N\rightarrow\infty}
\frac{|E\cap\mathcal{M}_N|}{|\mathcal{M}_N|}
&\leq
\limsup_{N\rightarrow\infty}
\frac{|E_t\cap\mathcal{M}_N|}{|\mathcal{M}_N|}
-\liminf_{N\rightarrow \infty}
\frac{|(E_t\setminus E)\cap\mathcal{M}_N|}{|\mathcal{M}_N|}\\
&\leq D(E_t),
\end{align*}
for all $t\in\N$. Let $t$ tend to $\infty$, and recall that the
number of monic irreducible polynomials with degree $m$ in $\F[x]$
is $\varphi_m$. From Lemma \ref{finite} and the definition of
$q$-zeta function, we can conclude that
\begin{align*}
\lim_{N\rightarrow \infty}\frac{|E\cap
\mathcal{M}_N|}{|\mathcal{M}_N|}&=\lim_{t\rightarrow\infty}
D(E_t)\\
&=\lim_{t\rightarrow\infty}\prod_{j=n-k+1}^{n}\prod_{f\in
\mathcal{P}_t}(1-q^{-j\deg(f)})\\
&=\lim_{t\rightarrow\infty}\prod_{j=n-k+1}^{n}
\prod_{m=1}^{t}(1-q^{-jm})^{\varphi_m}\\
&=\prod_{j=n-k+1}^{n}
\prod_{m=1}^{+\infty}(1-q^{-jm})^{\varphi_m}
=\prod_{j=n-k+1}^{n}\zeta_q(j)^{-1}.
\end{align*}
This completes the proof of Theorem \ref{result}. \ \ \ \ \ \ \ \ \
\ \ \ \ \ \ \ \ \ \ \ \ \ \ \ \ \ \ \ \ \ \ \ \ \ \ \ \ \ \ \ \ \ \
\ \ \ \ \ \ \ \  $\Box$

\vskip10pt

\noindent{\small\textbf {Acknowledgments}  \ \ The authors are very grateful to the anonymous referee for his/her careful comments.
The work of X. G. was partially supported by NSFC (No. 11026155 and No. 11101380) and that of G. Y. was partially supported by NSFC (No. 11201431) and Foundations of Zhengzhou University.}


\begin{thebibliography}{99}

\bibitem{BeBe} A. T. Benjamin and C. D. Bennett, The Probability of Relatively Prime Polynomials. {\it Mathematics Magazine}, 80 (3), 2007, 196-202.

\bibitem{Di} G. L. Dirichlet, \"{U}ber die Bestimmung der mittleren Werthe in der Zahlentheorie. Abhandlungen K\"{o}niglich Preuss. Akad. Wiss., 1849, 69-83; G. Lejeune Dirichlet's Werke. II, Chelsea, 1969, 49-66.

\bibitem{Kac} M. Kac, {\it Statistical independence in
 probability, analysis and number theory.} The Carus Mathematical Monographs, No. 12. Published by  Mathematical Association of America. Distributed by John Wiley and Sons, Inc., New York, 1959.

\bibitem{Ku} J. Kubilius, {\it Probabilistic methods in the theory of numbers.} Translations of Mathematical Monographs, Vol. 11, American Mathematical Society, 1964.

\bibitem{KuSu} H. Kubota and H. Sugita, Probabilistic proof of limit theorems in number theory by means of adeles, {\it Kyushu J. Math.}, 56, 2002, 391-404.
\bibitem{LiNi} R. Lidl and H. Niederreiter, {\it Finite fields}. Cambridge University Press, Cambridge, 1997.

\bibitem{MaRoWa} G. Maze, J. Rosenthal and U. Wagner, Natural density of rectangular unimodular integer matrices. {\it Liner Algebra and its Applications}, 434 (5), 2011, 1319-1324.

\bibitem{Mo} K. E. Morrison, Random Polynomials over Finite Fields. {\href{http://www.calpoly.edu/$\sim$kmorriso/Research/RPFF.pdf}
    {\it http://www.calpoly.edu/$\sim$kmorriso/Research/RPFF.pdf}}, 1999.

\bibitem{Qu} D. Quillen, Projective modules over polynomial rings. {\it Invent. Math.}, 36, 1976, 167-171.

\bibitem{SuTa} H. Sugita and S. Takanobu, The probability of two $\F_q$-polynomials to be coprime. {\it Advanced Studies in Pure Mathematics}, 49, 2007, 455-478.

\bibitem{Su} A. A. Suslin, Projective modules over polynomial rings are free. {\it Soviet Math.},  17 (4), 1976, 1160-1164.

\bibitem{Te} G. Tenenbaum,  {\it Introduction to analytic and probabilistic number theory}. Cambridge Studies in Advanced Mathematics, Vol. 46, Cambridge University Press, Cambridge, 1995.

\bibitem{YoPi} D. C. Youla and P. F. Pickel, The Quillen-Suslin theorem
and the structure of $n$-dimensional elementary polynomial matrices.
{\it IEEE Trans. Circuits Systems}, 31 (6), 1984, 513-518.

\bibitem{Web} {\href{http://wwwb.math.rwth-aachen.de/QuillenSuslin/}
    {\it http://wwwb.math.rwth-aachen.de/QuillenSuslin}}

\end{thebibliography}
\end{document}